\documentclass[reqno,centertags]{amsart}
\usepackage{amssymb}
\usepackage{amsfonts}

\newtheorem{theorem}{Theorem}
\theoremstyle{plain}

\newtheorem{corollary}{Corollary}

\numberwithin{equation}{section}

\input{tcilatex}

\begin{document}
\title[Further refinements and reverses of the young and Heinz inequalities]{%
Further generalizations, refinements, and reverses of the young and Heinz
inequalities}
\author{Yousef Al-Manasrah}
\address{Department of Mathematics, Al-Zaytoonah University of Jordan}
\email{y.manasrah@zuj.edu.jo}
\author{Fuad Kittaneh}
\address{Department of Mathematics, The University of Jordan, Amman, Jordan}
\email{fkitt@ju.edu.jo}
\date{}
\subjclass[2010]{ 15A60, 26A51, 26D20}
\keywords{ Convex function, Young inequality, Heinz inequality, positive \ \\
\ semidefinite matrix, unitarily invariant norm. }

\begin{abstract}
In this paper, we give a new inequality for convex functions of real
variables, and we apply this inequality to obtain considerable
generalizations, refinements, and reverses of the Young and Heinz
inequalities for positive scalars. Applications to unitarily invariant norm
inequalities involving positive semidefinite matrices are also given.
\end{abstract}

\maketitle

\section{Introduction}

The numerical Young's inequality for positive real numbers says that%
\begin{equation}
a^{\nu }b^{1-\nu }\leq \nu a+(1-\nu )b,  \label{ineq 1}
\end{equation}%
where $a,b>0$\ and $\ 0\leq \nu \leq 1$. Equivalently,%
\begin{equation*}
ab\leq \frac{a^{p}}{p}+\frac{b^{q}}{q},
\end{equation*}%
where $p,q>0$ and $\frac{1}{p}+\frac{1}{q}=1$.

Kittaneh and Manasrah in \cite{R6} and \cite{R7}, respectively, refined the
inequality (\ref{ineq 1}) and gave a reverse of it in the following forms:%
\begin{equation}
a^{\nu }b^{1-\nu }+r_{0}\left( \sqrt{a}-\sqrt{b}\right) ^{2}\leq \nu
a+(1-\nu )b  \label{ineq 2}
\end{equation}%
and 
\begin{equation}
\nu a+(1-\nu )b\leq a^{\nu }b^{1-\nu }+R_{0}\left( \sqrt{a}-\sqrt{b}\right)
^{2},  \label{ineq 3}
\end{equation}%
where $r_{0}=\min \left\{ \nu ,1-\nu \right\} $\textit{\ and }$R_{0}=\max
\left\{ \nu ,1-\nu \right\} $.

For our purpose in this paper, the inequalities (\ref{ineq 2}) and (\ref%
{ineq 3}) are combined and expressed so that%
\begin{equation}
r_{0}\left( (a+b)-2\sqrt{ab}\right) \leq \nu a+(1-\nu )b-a^{\nu }b^{1-\nu
}\leq R_{0}\left( (a+b)-2\sqrt{ab}\right) .  \label{ineq 4}
\end{equation}

In \cite{R4}, Hirzallah and Kittaneh proved that if $a,b>0$\ and $\ 0\leq
\nu \leq 1$, then%
\begin{equation}
\left( a^{\nu }b^{1-\nu }\right) ^{2}+r_{0}^{2}\left( a-b\right) ^{2}\leq
(\nu a+(1-\nu )b)^{2},  \label{ineq 5}
\end{equation}%
where $r_{0}=\min \left\{ \nu ,1-\nu \right\} .$

In \cite{R7}, Kittaneh and Manasrah gave a reverse of the inequality (\ref%
{ineq 5}) in the following form:%
\begin{equation}
\left( \nu a+(1-\nu )b\right) ^{2}\leq \left( a^{\nu }b^{1-\nu }\right)
^{2}+R_{0}^{2}\left( a-b\right) ^{2},  \label{ineq 6}
\end{equation}%
where $R_{0}=\max \left\{ \nu ,1-\nu \right\} $.

Also, for our purpose in this paper, the inequalities (\ref{ineq 5}) and (%
\ref{ineq 6}) are combined and expressed so that%
\begin{equation}
r_{0}^{2}\left( (a+b)^{2}-4ab\right) \leq \left( \nu a+(1-\nu )b\right)
^{2}-\left( a^{\nu }b^{1-\nu }\right) ^{2}\leq R_{0}^{2}\left(
(a+b)^{2}-4ab\right) .  \label{ineq 7}
\end{equation}

Recently, the authors \cite{R1} proved the following theorem.

\begin{theorem}
\bigskip \label{T1.1}If $a,b>0$\ and $\ 0\leq \nu \leq 1$, then for $%
m=1,2,3,...$, we have 
\begin{equation}
\left( a^{\nu }b^{1-\nu }\right) ^{m}+r_{0}^{m}\left( a^{\frac{m}{2}}-b^{%
\frac{m}{2}}\right) ^{2}\leq \left( \nu a+(1-\nu )b\right) ^{m},
\label{ineq 8}
\end{equation}%
where $r_{0}=\min \left\{ \nu ,1-\nu \right\} .$
\end{theorem}

In fact, this is a generalization of the\ inequalities (\ref{ineq 2}) and (%
\ref{ineq 5}), which correspond to the cases $m=1$ and $m=2$, respectively.

The Heinz means are defined as%
\begin{equation*}
H_{\nu }(a,b)=\frac{a^{\nu }b^{1-\nu }+a^{1-\nu }b^{\nu }}{2}
\end{equation*}%
for $a,b>0$ and $0\leq \nu \leq 1$. These interesting means interpolate
between the geometric and arithmetic means. In fact, the Heinz inequalities
assert that%
\begin{equation*}
\sqrt{ab}\leq H_{\nu }(a,b)\leq \frac{a+b}{2}.
\end{equation*}%
Interchaning a and b in the inequalities (\ref{ineq 4}), and adding the
resulting inequalities to the inequalities (\ref{ineq 4}), we have

\begin{equation}
2r_{0}\left( a+b-2\sqrt{ab}\right) \leq a+b-\left( a^{\nu }b^{1-\nu
}+a^{1-\nu }b^{\nu }\right) \leq 2R_{0}\left( a+b-2\sqrt{ab}\right) .
\label{ineq 9}
\end{equation}%
Equivalently,\ \ \ \ \ \ \ \ \ \ \ \ \ \ \ \ \ \ \ \ \ \ \ \ \ \ \ \ \ \ \ \
\ \ \ \ \ \ \ \ \ \ \ \ \ \ \ \ \ \ \ \ \ \ \ \ \ \ \ \ \ \ \ \ \ \ \ \ \ \
\ \ \ \ \ \ \ \ \ \ \ \ \ \ \ \ \ \ \ \ \ \ \ \ \ \ \ \ \ \ \ \ \ \ \ \ \ \
\ \ \ \ \ \ \ \ \ \ \ \ \ \ \ \ \ \ \ \ \ \ \ \ \ \ \ \ \ \ \ \ \ \ \ \ \ \
\ \ \ \ \ \ \ \ \ \ \ \ \ \ \ \ \ \ \ \ \ \ \ \ \ \ \ \ \ \ \ \ \ \ \ \ \ \
\ \ \ \ \ \ \ \ \ \ \ \ \ \ \ \ \ \ \ \ 

\begin{equation}
2r_{0}\left( \frac{a+b}{2}-\sqrt{ab}\right) \leq \frac{a+b}{2}-H_{\nu
}(a,b)\leq 2R_{0}\left( \frac{a+b}{2}-\sqrt{ab}\right) ,  \label{ineq 10}
\end{equation}%
where $r_{0}=\min \left\{ \nu ,1-\nu \right\} $\textit{\ and }$R_{0}=\max
\left\{ \nu ,1-\nu \right\} .$

In Section 2, we present a new inequality for convex functions of real
variables. We apply this inequality to obtain considerable generalizations,
refinements, and reverses of the Young and Heinz inequalities (\ref{ineq 2}%
)-(\ref{ineq 10}). Applications to unitarily invariant norm inequalities
involving positive semidefinite matrices are given in Section 3.

\section{ Main results}

Let $f$ be a convex function defined on an interval $I$. If $x$, $y,z$ and $%
w $ are points in $I$ such that $w<z<y<x$, then it follows by a slope
argument that%
\begin{equation}
\frac{f(z)-f(w)}{z-w}\leq \frac{f(y)-f(w)}{y-w}\leq \frac{f(y)-f(z)}{y-z}%
\leq \frac{f(x)-f(y)}{x-y}  \label{ineq 11}
\end{equation}%
(see, e.g., \cite[p. 21]{R8}).

Our first result is a consequence of the inequalities (\ref{ineq 11}) for
convex functions of real variables.

\begin{theorem}
\label{T2.1}Let $\phi $ be a stricly increasing convex function defined on
an interval $I$. If $x,y,z,$ and $w$ are points in $I$ such that $\ $%
\begin{equation}
\text{ \ \ }z-w\leq x-y,  \label{ineq 12}
\end{equation}%
where $w\leq z\leq x$ and $y\leq x$, then 
\begin{equation}
(0\leq )\text{ \ \ }\phi (z)-\phi (w)\leq \phi (x)-\phi (y).  \label{ineq 13}
\end{equation}
\end{theorem}

\begin{proof}
First of all, observe that if $x=y$, then $z=w$,\ and so the inequality (\ref%
{ineq 13}) becomes an equality. If $y=w$ or $z=w$,\ then the inequality (\ref%
{ineq 13}) holds. Also, if $x=z$, then according to the inequality (\ref%
{ineq 12}), we have $w\geq y$, so $\phi (w)\geq \phi (y)$, and hence $\phi
(z)-\phi (w)\leq \phi (x)-\phi (y)$.\ \ \ 

Assume that $x\neq y$, $y\neq w$, $z\neq w$, and $x\neq z$. Then we have
three cases for ordering the points $x,y,z,w$\ \ as follows:

Case 1: $w<y\leq z<x$

Case 2: $w<z<y<x$

Case 3: $y<w<z<x$.

Now,\ if $y=z$ , then the case 1 becomes $w<y=z<x$, and so by the
inequalities (\ref{ineq 11}), we have $\frac{\phi (z)-\phi (w)}{z-w}\leq 
\frac{\phi (x)-\phi (y)}{x-y}$,\ which implies that $\phi (z)-\phi (w)\leq
\phi (x)-\phi (y)$. Suppose $y\neq z$. Then\ apply the inequalities (\ref%
{ineq 11}) to the cases 1 and 2 to get the inequality (\ref{ineq 13}).\ \ \
\ \ \ \ \ \ \ \ \ \ \ \ \ \ \ \ \ \ \ \ \ \ \ \ \ \ \ \ \ \ \ \ \ \ \ \ \ \
\ \ \ \ \ \ \ \ \ \ \ \ \ \ \ \ \ \ \ \ \ \ \ \ \ \ \ \ \ \ \ \ \ \ \ \ \ \
\ \ \ \ \ \ \ \ \ \ \ \ \ \ \ \ \ \ \ \ \ \ \ \ \ \ \ \ \ \ \ \ \ \ \ \ \ \
\ \ \ \ \ \ \ \ \ \ \ \ \ \ \ \ \ \ \ \ \ \ \ \ \ \ \ \ \ \ \ \ \ \ \ \ \ \
\ \ \ \ \ \ \ \ \ \ \ \ \ \ \ \ \ \ \ \ \ \ \ \ \ \ \ \ \ \ \ \ \ \ \ \ \ \
\ \ \ \ \ \ \ \ \ \ \ \ \ \ \ \ \ \ \ \ \ \ \ \ \ \ \ \ \ \ \ \ \ \ \ \ \ \
\ \ \ \ \ \ \ \ \ \ \ \ \ \ \ \ \ \ \ \ \ \ \ \ \ \ \ \ \ \ \ \ \ \ \ \ \ \
\ \ \ \ \ \ \ \ \ \ \ \ \ \ \ \ \ \ \ \ \ \ \ \ \ \ \ \ \ \ \ \ \ \ \ \ \ \
\ \ \ \ \ \ \ \ \ \ \ \ \ \ \ \ \ \ \ \ \ \ \ \ \ \ \ \ \ \ \ \ \ \ \ \ \ \
\ \ \ \ \ \ \ \ \ \ \ \ \ \ \ \ \ \ \ \ \ \ \ \ \ \ \ \ \ \ \ \ \ \ \ \ \ \
\ \ \ \ \ \ \ \ \ \ \ \ \ \ \ \ \ \ \ \ \ \ \ \ \ \ \ \ \ \ \ \ \ \ \ \ \ \
\ \ \ \ \ \ \ \ \ \ \ \ \ \ \ \ \ \ \ \ \ \ \ \ \ \ \ \ \ \ \ \ 

To discuss the third case, we apply the strictly increasing property of the
function $\phi $ to the sequence of points $y<w<z<x$, so we have $\phi
(y)<\phi (w)<\phi (z)<\phi (x)$ and this implies that $\phi (y)-\phi
(w)<0<\phi (x)-\phi (z)$, and hence $\phi (z)-\phi (w)<\phi (x)-\phi (y)$.

Thus, from the discussion above we have\ \ \ \ \ \ \ \ \ \ \ \ \ \ \ \ \ \ \
\ \ \ \ \ \ \ \ \ \ \ \ \ \ \ \ \ \ \ \ \ \ \ \ \ \ \ \ \ \ \ \ \ \ \ \ \ \
\ \ \ \ \ \ \ \ \ \ \ \ \ \ \ \ \ \ \ \ \ \ \ \ \ \ \ \ \ \ \ \ \ \ \ \ \ \
\ \ \ \ \ \ \ \ \ \ \ \ \ \ \ \ \ \ \ \ \ \ \ \ \ \ \ \ \ \ \ \ \ \ \ \ \ \
\ \ \ \ \ \ \ \ \ \ \ \ \ \ \ \ \ \ \ \ \ \ \ \ \ \ \ \ \ \ \ \ \ \ \ \ \ \
\ \ \ \ \ \ \ \ \ \ \ \ \ \ \ \ \ \ \ \ \ 
\begin{equation*}
\phi (z)-\phi (w)\leq \phi (x)-\phi (y).
\end{equation*}%
This completes the proof.
\end{proof}

As a direct consequence of Theorem \ref{T2.1}, we have the following
generalizations of the inequalities (\ref{ineq 4}).

\begin{corollary}
\label{C2.1} Let $\phi :[0,\infty )\rightarrow 
\mathbb{R}
$ be a strictly increasing convex function. If $a,b>0$\ and $\ 0\leq \nu
\leq 1$, then we have%
\begin{eqnarray}
\phi \left( r_{0}\left( a+b\right) \right) -\phi \left( 2r_{0}\sqrt{ab}%
\right) &\leq &\phi \left( \nu a+(1-\nu )b\right) -\phi \left( a^{\nu
}b^{1-\nu }\right)  \notag \\
&\leq &\phi \left( R_{0}\left( a+b\right) \right) -\phi \left( 2R_{0}\sqrt{ab%
}\right) ,  \label{ineq 14}
\end{eqnarray}%
where $r_{0}=\min \left\{ \nu ,1-\nu \right\} $\textit{\ and }$R_{0}=\max
\left\{ \nu ,1-\nu \right\} $.
\end{corollary}

\begin{proof}
Let $x=\nu a+(1-\nu )b$, \ \ $y=a^{\nu }b^{1-\nu },$ \ \ $z=r_{0}\left(
a+b\right) $, $w=2r_{0}\sqrt{ab}$\ , $z^{\prime }=R_{0}\left( a+b\right) $,
and $w^{\prime }=2R_{0}\sqrt{ab}$. Then based on the inequalities (\ref{ineq
4}) and the arithmetic-geometric mean inequality, we have 
\begin{equation*}
z-w\leq x-y\leq z^{\prime }-w^{\prime }.
\end{equation*}%
The first and the second inequalities in (\ref{ineq 14}) follow directly by
applying Theorem \ref{T2.1} to the inequalities $z-w\leq x-y$, with $w\leq
z\leq x$, $y\leq x$ and $x-y\leq z^{\prime }-w^{\prime }$ with $y\leq x\leq
z^{\prime }$, $w^{\prime }\leq z^{\prime }$, respectively. This completes
the proof.
\end{proof}

A particular case of Corollary \ref{C2.1}, which is obtained by taking $\phi
(x)=x^{p}$ $(p\in 
\mathbb{R}
$, $p\geq 1)$ can be stated as follows.

\begin{corollary}
\label{C2.2}If $a,b>0$\ and $\ 0\leq \nu \leq 1$, then for $p\in 
\mathbb{R}
$, $p\geq 1$, we have%
\begin{eqnarray}
r_{0}^{p}\left( (a+b)^{p}-\left( 2\sqrt{ab}\right) ^{p}\right) &\leq &(\nu
a+(1-\nu )b)^{p}-\left( a^{\nu }b^{1-\nu }\right) ^{p}  \notag \\
&\leq &R_{0}^{p}\ \left( (a+b)^{p}-\left( 2\sqrt{ab}\right) ^{p}\right) ,
\label{ineq 15}
\end{eqnarray}%
where $r_{0}=\min \left\{ \nu ,1-\nu \right\} $\textit{\ and }$R_{0}=\max
\left\{ \nu ,1-\nu \right\} $. \ \ \ \ \ \ \ \ \ \ \ \ \ \ \ \ \ \ \ \ \ \ \
\ \ \ \ \ \ \ \ \ \ \ \ \ \ \ \ \ \ \ \ \ \ \ \ \ \ \ \ \ \ \ \ \ \ \ \ \ \
\ \ \ \ \ \ \ \ \ \ \ \ \ \ \ \ \ \ \ \ \ \ \ \ \ \ \ \ \ \ \ \ \ \ \ \ \ \
\ \ \ \ \ \ \ \ \ \ \ \ \ \ \ \ \ \ \ \ \ \ \ \ \ \ \ \ \ \ \ \ \ \ \ \ \ \
\ \ \ \ \ \ \ \ \ \ \ \ \ \ \ \ \ \ \ \ \ \ \ \ \ \ \ \ \ \ \ \ \ \ \ \ \ \
\ \ \ \ \ \ \ \ \ \ \ \ \ \ \ \ \ \ \ \ \ \ \ \ \ \ \ \ \ \ \ \ \ \ \ \ \ \
\ 
\end{corollary}

One can observe that the inequalities (\ref{ineq 15}) are reduced to the
inequalities (\ref{ineq 4}) and (\ref{ineq 7}) when $p=1$ and $p=2$,
respectively.

The next theorem demonstrates the relationship between\ the inequalities (%
\ref{ineq 15}) and (\ref{ineq 8}). In fact, it confirms that the first
inequality in (\ref{ineq 15}) is uniformly better than the inequality (\ref%
{ineq 8}), and the second inequality in (\ref{ineq 15}) provides a reverse
of the same inequality. It should be noted here that the variable $p$ \ in
the inequalities (\ref{ineq 15}) is continuous $(p\in 
\mathbb{R}
$, $p\geq 1)$, while the variable $m$ in the inequalitiy (\ref{ineq 8}) is
discrete $(m=1,2,3,...)$.

\begin{theorem}
\label{T2.2}If $a,b>0$\ and $\ 0\leq \nu \leq 1$, then we have%
\begin{eqnarray}
r_{0}^{m}\left( a^{\frac{m}{2}}-b^{\frac{m}{2}}\right) ^{2} &\leq
&r_{0}^{m}\left( \left( a+b\right) ^{m}-\left( 2\sqrt{ab}\right) ^{m}\right)
\notag \\
&\leq &\left( \nu a+(1-\nu )b\right) ^{m}-\left( a^{\nu }b^{1-\nu }\right)
^{m}  \notag \\
\text{ \ \ \ \ \ \ \ \ \ \ \ \ \ \ \ \ \ \ \ \ \ \ \ \ \ \ \ \ \ \ \ \ \ \ \
\ \ \ \ \ \ \ \ \ \ } &\leq &R_{0}^{m}\left( \left( a+b\right) ^{m}-\left( 2%
\sqrt{ab}\right) ^{m}\right)  \label{ineq 16}
\end{eqnarray}%
for $m=1,2,3,...$, where $r_{0}=\min \left\{ \nu ,1-\nu \right\} $.
\end{theorem}

\begin{proof}
It is clear that the second and third inequalities in (\ref{ineq 16}) are
special cases of the inequalities (\ref{ineq 15}). So, it is enough to prove
the first inequality in (\ref{ineq 16}).

To do this, observe that the cases $m=1$ and $m=2$ degenerate to an
equality. For $m=3,4,5,...$, we discuss two cases.

Case 1: If $m$ is even, we have%
\begin{eqnarray*}
(a+b)^{m} &=&a^{m}+\sum_{i=1}^{\frac{m}{2}-1}\left( 
\begin{array}{c}
m \\ 
i%
\end{array}%
\right) \left( a^{i}b^{m-i}+a^{m-i}b^{i}\right) +\left( 
\begin{array}{c}
m \\ 
\frac{m}{2}%
\end{array}%
\right) a^{\frac{m}{2}}b^{\frac{m}{2}}+b^{m} \\
&\geq &a^{m}+2a^{\frac{m}{2}}b^{\frac{m}{2}}\sum_{i=1}^{\frac{m}{2}-1}\left( 
\begin{array}{c}
m \\ 
i%
\end{array}%
\right) +\left( 
\begin{array}{c}
m \\ 
\frac{m}{2}%
\end{array}%
\right) a^{\frac{m}{2}}b^{\frac{m}{2}}+b^{m}\text{ \ } \\
&&\text{ \ \ \ \ \ \ \ \ \ \ \ \ \ \ \ \ \ \ \ \ \ \ \ \ \ (by the
arithmetic-geometric mean inequality)} \\
&=&a^{m}+a^{\frac{m}{2}}b^{\frac{m}{2}}\left( 2\sum_{i=1}^{\frac{m}{2}%
-1}\left( 
\begin{array}{c}
m \\ 
i%
\end{array}%
\right) +\left( 
\begin{array}{c}
m \\ 
\frac{m}{2}%
\end{array}%
\right) \right) +b^{m} \\
&=&a^{m}+b^{m}+a^{\frac{m}{2}}b^{\frac{m}{2}}\sum_{i=1}^{m-1}\left( 
\begin{array}{c}
m \\ 
i%
\end{array}%
\right) \\
&=&a^{m}+b^{m}+\left( 2^{m}-2\right) a^{\frac{m}{2}}b^{\frac{m}{2}}.
\end{eqnarray*}

Case 2: If $m$ is odd, we have%
\begin{eqnarray*}
(a+b)^{m} &=&a^{m}+\sum_{i=1}^{\frac{m-1}{2}}\left( 
\begin{array}{c}
m \\ 
i%
\end{array}%
\right) \left( a^{i}b^{m-i}+a^{m-i}b^{i}\right) +b^{m} \\
&\geq &a^{m}+b^{m}+2a^{\frac{m}{2}}b^{\frac{m}{2}}\sum_{i=1}^{\frac{m-1}{2}%
}\left( 
\begin{array}{c}
m \\ 
i%
\end{array}%
\right) \text{ } \\
&&\text{ \ \ \ \ \ \ \ \ \ \ \ \ \ \ \ \ \ (by the arithmetic-geometric mean
inequality)} \\
&=&a^{m}+b^{m}+a^{\frac{m}{2}}b^{\frac{m}{2}}\sum_{i=1}^{m-1}\left( 
\begin{array}{c}
m \\ 
i%
\end{array}%
\right) \\
&=&a^{m}+b^{m}+\left( 2^{m}-2\right) a^{\frac{m}{2}}b^{\frac{m}{2}}.
\end{eqnarray*}%
Now,

for $m=3,4,5,...,$we have%
\begin{eqnarray*}
\left( a+b\right) ^{m}-\left( 2\sqrt{ab}\right) ^{m} &=&\left( a+b\right)
^{m}-2^{m}a^{\frac{m}{2}}b^{\frac{m}{2}} \\
&\geq &a^{m}+b^{m}-2a^{\frac{m}{2}}b^{\frac{m}{2}} \\
&=&\left( a^{\frac{m}{2}}-b^{\frac{m}{2}}\right) ^{2}.
\end{eqnarray*}%
Thus, 
\begin{equation*}
r_{0}^{m}\left( a^{\frac{m}{2}}-b^{\frac{m}{2}}\right) ^{2}\leq
r_{0}^{m}\left( \left( a+b\right) ^{m}-\left( 2\sqrt{ab}\right) ^{m}\right)
\leq \left( \nu a+(1-\nu )b\right) ^{m}-\left( a^{\nu }b^{1-\nu }\right) ^{m}
\end{equation*}%
for $m=1,2,3,...$ This completes the proof.\ \ \ \ \ \ \ \ \ \ \ \ \ \ \ \ \
\ \ \ \ \ \ \ \ \ \ \ \ \ \ \ \ \ \ \ \ \ \ \ \ \ \ \ \ \ \ \ \ \ \ \ \ \ \
\ \ \ \ \ \ \ \ \ \ \ \ \ \ \ \ \ \ \ \ \ \ \ \ \ \ \ \ \ \ \ \ \ \ \ \ \ \
\ \ \ \ \ \ \ \ \ \ \ \ \ \ \ \ \ \ \ \ \ \ \ \ \ \ \ \ \ \ \ \ \ \ \ \ \ \
\ \ \ \ \ \ \ \ \ \ \ \ \ \ \ \ \ \ \ \ \ \ \ \ \ \ \ \ \ \ \ \ \ \ \ \ \ \
\ \ 
\end{proof}

Applying Theorem \ref{T2.1} again, we have the following generalizations of
the inequalities (\ref{ineq 9}) and (\ref{ineq 10}), respectively:%
\begin{eqnarray}
\phi \left( 2r_{0}\left( a+b\right) \right) -\phi \left( 4r_{0}\sqrt{ab}%
\right) &\leq &\phi \left( a+b\right) -\phi \left( a^{\nu }b^{1-\nu
}+a^{1-\nu }b^{\nu }\right)  \notag \\
&\leq &\phi \left( 2R_{0}\left( a+b\right) \right) -\phi \left( 4R_{0}\sqrt{%
ab}\right)  \label{ineq 17}
\end{eqnarray}%
and%
\begin{eqnarray*}
\phi \left( 2r_{0}\left( \frac{a+b}{2}\right) \right) -\phi \left( 2r_{0}%
\sqrt{ab}\right) &\leq &\phi \left( \frac{a+b}{2}\right) -\phi \left( H_{\nu
}(a,b)\right) \\
&\leq &\phi \left( 2R_{0}\left( \frac{a+b}{2}\right) \right) -\phi \left(
2R_{0}\sqrt{ab}\right) .\text{ \ \ \ \ \ \ \ }
\end{eqnarray*}%
In particular, if $\phi (x)=x^{p}$ $(p\in 
\mathbb{R}
,$ $p\geq 1),$we have%
\begin{eqnarray}
(2r_{0})^{p}\left( \left( a+b\right) ^{p}-\left( 2\sqrt{ab}\right)
^{p}\right) &\leq &\left( a+b\right) ^{p}-\left( a^{\nu }b^{1-\nu }+a^{1-\nu
}b^{\nu }\right) ^{p}  \notag \\
&\leq &\text{\ }(2R_{0})^{p}\left( \left( a+b\right) ^{p}-\left( 2\sqrt{ab}%
\right) ^{p}\right) .  \label{ineq 18}
\end{eqnarray}

\section{\protect\bigskip Some inequalities for unutarily invariant norms}

In this section, we obtain matrix versions of the scalar inequalities
presented in Sections 1 and 2.

Let $M_{n}(%
\mathbb{C}
)$ be the space of $n\times n$ complex matrices. A norm $\left\vert
\left\vert \left\vert .\right\vert \right\vert \right\vert $ on $M_{n}(%
\mathbb{C}
)$ is called unitarily invariant if $\left\vert \left\vert \left\vert
UAV\right\vert \right\vert \right\vert =\left\vert \left\vert \left\vert
A\right\vert \right\vert \right\vert $ for all $A\in M_{n}(%
\mathbb{C}
)$ and for all unitary matrices $U,V$ $\in M_{n}(%
\mathbb{C}
)$. An example of unitarly invariant norms is the Schatten $p$-norm, denoted
by $\left\Vert .\right\Vert _{p}$ and defined, for $1\leq p<\infty $, by%
\begin{equation*}
\left\Vert A\right\Vert _{p}=\left( \sum_{i=1}^{n}s_{i}^{p}(A)\right) ^{%
\frac{1}{p}},
\end{equation*}%
where $s_{1}(A)\geq s_{2}(A)\geq ...\geq s_{n}(A)$ are the singular values
of $A\in M_{n}(%
\mathbb{C}
)$. The Schatten $1$-norm of $A$ is the trace norm, which can be expressed
as $\left\Vert A\right\Vert _{1}=$ tr $\left\vert A\right\vert $. The
Schatten $2-$norm of $A=\left[ a_{ij}\right] $ is known as the
Hilbert-Schmidt (or the Frobenius) norm, which can be expressed as%
\begin{equation*}
\left\Vert A\right\Vert _{2}=\left( \sum_{i,j=1}^{n}\left\vert
a_{ij}\right\vert ^{2}\right) ^{\frac{1}{2}}.
\end{equation*}%
Another important example of unitarily invariant norms on $M_{n}(%
\mathbb{C}
)$ is the spectral (or the usual operator) norm $\left\Vert .\right\Vert $,
given by%
\begin{equation*}
\left\Vert A\right\Vert =s_{1}(A).
\end{equation*}

Based on the refined and reversed Young inequalities (\ref{ineq 5}) and (\ref%
{ineq 6}), Hirzallah and Kittaneh \cite{R4}, and Kittaneh and Manasrah \cite%
{R7}, respectively, proved that if $A,B,X\in M_{n}(%
\mathbb{C}
)$ such that $A$ and $B$ are positive semidefinite, then 
\begin{equation}
r_{0}^{2}\left\Vert AX-XB\right\Vert _{2}^{2}\leq \left\Vert \nu AX+(1-\nu
)XB\right\Vert _{2}^{2}-\left\Vert A^{\nu }XB^{1-\nu }\right\Vert _{2}^{2}
\label{ineq 19}
\end{equation}%
and%
\begin{equation}
\left\Vert \nu AX+(1-\nu )XB\right\Vert _{2}^{2}-\left\Vert A^{\nu
}XB^{1-\nu }\right\Vert _{2}^{2}\leq R_{0}^{2}\left\Vert AX-XB\right\Vert
_{2}^{2},  \label{ineq 20}
\end{equation}%
where $0\leq \nu \leq 1$, $r_{0}=\min \left\{ \nu ,1-\nu \right\} $, \textit{%
and }$R_{0}=\max \left\{ \nu ,1-\nu \right\} .$

It can be easily shown that%
\begin{equation*}
\left\Vert AX-XB\right\Vert _{2}^{2}=\left\Vert AX+XB\right\Vert
_{2}^{2}-4\left\Vert A^{\frac{1}{2}}XB^{\frac{1}{2}}\right\Vert _{2}^{2}.
\end{equation*}%
Thus, the inequalities (\ref{ineq 19}) and (\ref{ineq 20}) can be combined
and expressed so that%
\begin{eqnarray}
r_{0}^{2}\left( \left\Vert AX+XB\right\Vert _{2}^{2}-4\left\Vert A^{\frac{1}{%
2}}XB^{\frac{1}{2}}\right\Vert _{2}^{2}\right) &\leq &\left\Vert \nu
AX+(1-\nu )XB\right\Vert _{2}^{2}-\left\Vert A^{\nu }XB^{\nu }\right\Vert
_{2}^{2}  \notag \\
&\leq &R_{0}^{2}\left( \left\Vert AX+XB\right\Vert _{2}^{2}-4\left\Vert A^{%
\frac{1}{2}}XB^{\frac{1}{2}}\right\Vert _{2}^{2}\right) .  \notag \\
&&  \label{ineq 21}
\end{eqnarray}

Applying Theorem \ref{T2.1} to the inequalities (\ref{ineq 21}), the
following general result holds.

\begin{theorem}
\label{T3.1}Let $A,B,X\in M_{n}(%
\mathbb{C}
)$\ such that $A$\ and $B$\ are positive semidefinite. If $\ \phi :[0,\infty
)\rightarrow 
\mathbb{R}
$ is a strictly increasing convex function, then we have%
\begin{eqnarray*}
&&\phi \left( r_{0}^{2}\left\Vert AX+XB\right\Vert _{2}^{2}\right) -\phi
\left( 4r_{0}^{2}\left\Vert A^{\frac{1}{2}}XB^{\frac{1}{2}}\right\Vert
_{2}^{2}\right) \\
&\leq &\phi \left( \left\Vert \nu AX+(1-\nu )XB\right\Vert _{2}^{2}\right)
-\phi \left( \left\Vert A^{\nu }XB^{\nu }\right\Vert _{2}^{2}\right) \\
&\leq &\phi \left( R_{0}^{2}\left\Vert AX+XB\right\Vert _{2}^{2}\right)
-\phi \left( 4R_{0}^{2}\left\Vert A^{\frac{1}{2}}XB^{\frac{1}{2}}\right\Vert
_{2}^{2}\right) ,
\end{eqnarray*}%
where $r_{0}=\min \left\{ \nu ,1-\nu \right\} $\textit{\ and }$R_{0}=\max
\left\{ \nu ,1-\nu \right\} $.
\end{theorem}

In particular, when $\phi (x)=x^{\frac{p}{2}}$ $(p\in 
\mathbb{R}
$ and $p\geq 2)$, we have%
\begin{eqnarray*}
r_{0}^{p}\left( \left\Vert AX+XB\right\Vert _{2}^{p}-2^{p}\left\Vert A^{%
\frac{1}{2}}XB^{\frac{1}{2}}\right\Vert _{2}^{p}\right) &\leq &\left\Vert
\nu AX+(1-\nu )XB\right\Vert _{2}^{p}-\left\Vert A^{\nu }XB^{\nu
}\right\Vert _{2}^{p} \\
&\leq &R_{0}^{p}\left( \left\Vert AX+XB\right\Vert _{2}^{p}-2^{p}\left\Vert
A^{\frac{1}{2}}XB^{\frac{1}{2}}\right\Vert _{2}^{p}\right) .
\end{eqnarray*}%
\ 

If $A,B,X\in M_{n}(%
\mathbb{C}
)$ such that $A$ and $B$ are positive semidefinite, then it is known that
for any unitarily invariant norm, the function $f(\nu )=\left\vert
\left\vert \left\vert A^{\nu }XB^{1-\nu }+A^{1-\nu }XB^{\nu }\right\vert
\right\vert \right\vert $ is convex on $[0,1]$ and attains its minimum at $%
\nu =\frac{1}{2}$ (see, e.g., \cite[p. 265]{R2}). \ \ \ \ \ \ \ \ \ \ \ \ \
\ \ \ \ \ \ \ \ \ \ \ \ \ \ \ \ \ \ \ \ \ \ \ \ \ \ \ \ \ \ \ \ \ \ \ \ \ \
\ \ \ \ \ \ \ \ \ \ \ \ \ \ \ \ \ \ \ \ \ \ \ \ \ \ \ \ \ \ \ \ \ \ \ \ \ \
\ \ \ \ \ \ \ \ \ \ \ \ \ \ \ \ \ \ \ \ \ \ \ \ \ \ \ \ \ \ \ \ \ \ \ \ \ \
\ \ \ \ \ \ \ \ \ \ \ \ \ \ \ \ \ \ \ \ \ \ \ \ \ \ \ \ \ \ \ \ \ \ \ \ \ \
\ \ \ \ \ \ \ \ \ \ \ \ \ \ \ \ \ \ \ \ \ \ \ \ \ \ \ \ \ \ \ \ \ \ \ \ \ \
\ \ \ \ \ \ \ \ \ \ \ \ \ \ \ \ \ \ \ \ \ \ \ \ \ \ \ 

Bhatia and Davis \cite{R3} proved that if $A,B,X\in M_{n}(%
\mathbb{C}
)$ such that $A$ and $B$ are positive semidefinite, then\ \ \ \ \ \ \ \ \ \
\ \ \ \ \ \ \ \ \ \ \ \ \ \ \ \ \ \ \ \ \ \ \ \ \ \ \ \ \ \ \ \ \ \ \ \ \ \
\ \ \ \ \ \ \ \ \ \ \ \ \ \ \ \ \ \ \ \ \ \ \ \ \ \ \ \ \ \ \ \ \ \ \ \ \ \
\ \ \ \ \ \ \ \ \ \ \ \ \ \ \ \ \ \ \ \ \ \ \ \ \ \ \ \ \ \ \ \ \ \ \ \ \ \
\ \ \ \ \ \ \ \ \ \ \ \ \ \ \ \ \ \ \ \ \ \ \ \ \ \ \ \ \ \ \ \ \ \ \ \ \ \
\ \ \ \ \ \ \ \ \ \ \ \ \ \ \ \ \ \ \ \ \ \ \ \ \ \ \ \ \ \ \ \ \ \ \ \ \ \
\ \ \ \ \ \ \ \ \ \ \ \ \ \ \ \ \ \ \ \ \ \ \ \ \ \ \ \ \ \ \ \ \ \ \ \ \ \
\ \ \ \ \ \ \ \ \ \ \ \ \ \ \ \ \ \ \ \ \ \ 
\begin{equation*}
2\left\vert \left\vert \left\vert A^{\frac{1}{2}}XB^{\frac{1}{2}}\right\vert
\right\vert \right\vert \leq \left\vert \left\vert \left\vert A^{\nu
}XB^{1-\nu }+A^{1-\nu }XB^{\nu }\right\vert \right\vert \right\vert \leq
\left\vert \left\vert \left\vert AX+XB\right\vert \right\vert \right\vert ,
\end{equation*}%
where $0\leq \nu \leq 1$. These inequalities are known as Heinz norm
inequalities.\ \ \ \ \ \ \ \ \ \ \ \ \ \ \ \ \ \ \ \ \ \ \ \ \ \ \ \ \ \ \ \
\ \ \ \ \ \ \ \ \ \ \ \ \ \ \ \ \ \ \ \ \ \ \ \ \ \ \ \ \ \ \ \ \ \ \ \ \ \
\ \ \ \ \ \ \ \ \ \ \ \ \ \ \ \ \ \ \ \ \ \ \ \ \ \ \ \ \ \ \ \ \ \ \ \ \ \
\ \ \ \ \ \ \ \ \ \ \ \ \ \ \ \ \ \ \ \ \ \ \ \ \ \ \ \ \ \ \ \ \ \ \ \ \ \
\ \ \ \ \ \ \ \ \ \ \ \ \ \ \ \ \ \ \ \ \ \ \ \ \ \ \ \ \ \ \ \ \ \ \ \ \ \
\ \ \ \ \ \ \ \ \ \ \ \ \ \ \ \ \ \ \ \ \ \ \ \ \ \ \ \ \ \ \ \ \ \ \ \ \ \
\ \ \ \ \ \ \ \ \ \ \ \ \ \ \ \ \ \ \ \ \ \ \ \ \ \ \ \ \ \ \ \ \ \ 

Kittaneh \cite{R5} proved that if $A,B,X\in M_{n}(%
\mathbb{C}
)$ such that $A$ and $B$ are positive semidefinite, then

\begin{equation}
2r_{0}\left( \left\vert \left\vert \left\vert AX+XB\right\vert \right\vert
\right\vert -2\left\vert \left\vert \left\vert A^{\frac{1}{2}}XB^{\frac{1}{2}%
}\right\vert \right\vert \right\vert \right) \leq \left\vert \left\vert
\left\vert AX+XB\right\vert \right\vert \right\vert -\left\vert \left\vert
\left\vert A^{\nu }XB^{1-\nu }+A^{1-\nu }XB^{\nu }\right\vert \right\vert
\right\vert ,  \label{ineq 22}
\end{equation}%
where $0\leq \nu \leq 1$ and $r_{0}=\min \{\nu ,1-\nu \}.$

In the next theorem, we obtain a reverse of the inequalitiy (\ref{ineq 22})
as follows.

\begin{theorem}
\label{T3.2} If $A,B,X\in M_{n}(%
\mathbb{C}
)$ such that $A$ and $B$ are positive semidefinite, then 
\begin{equation}
\left\vert \left\vert \left\vert AX+XB\right\vert \right\vert \right\vert
-\left\vert \left\vert \left\vert A^{\nu }XB^{1-\nu }+A^{1-\nu }XB^{\nu
}\right\vert \right\vert \right\vert \leq 2R_{0}\left( \left\vert \left\vert
\left\vert AX+XB\right\vert \right\vert \right\vert -2\left\vert \left\vert
\left\vert A^{\frac{1}{2}}XB^{\frac{1}{2}}\right\vert \right\vert
\right\vert \right) ,  \label{ineq 23}
\end{equation}%
where $\ 0\leq \nu \leq 1$ and $R_{0}=\max \{\nu ,1-\nu \}$.\ \ \ \ \ \ \ \
\ \ \ \ \ \ \ \ \ \ \ \ \ \ \ \ \ \ \ \ \ \ \ \ \ \ \ \ \ \ \ \ \ \ \ \ \ \
\ \ \ \ \ \ \ \ \ \ \ \ \ \ \ \ \ \ \ \ \ \ \ \ \ \ \ \ \ \ \ \ \ \ \ \ \ \
\ \ \ \ \ \ \ \ \ 
\end{theorem}

\begin{proof}
If $\nu =0$, $\nu =\frac{1}{2}$, or $\nu =1$, then the inequality (\ref{ineq
23}) is obviously true. Suppose that $0<\nu <1$, $\nu \neq \frac{1}{2}$.

If $0<\nu <\frac{1}{2}<1$, then based on the convexity of the function $%
f(\nu )=\left\vert \left\vert \left\vert A^{\nu }XB^{1-\nu }+A^{1-\nu
}XB^{\nu }\right\vert \right\vert \right\vert $, we have%
\begin{equation*}
\frac{f\left( \frac{1}{2}\right) -f\left( \nu \right) }{\frac{1}{2}-\nu }%
\leq \frac{f\left( 1\right) -f\left( \nu \right) }{1-\nu },
\end{equation*}%
and so \ \ \ \ \ \ \ \ \ \ \ \ \ \ \ \ \ \ \ \ \ \ \ \ \ \ \ \ \ \ \ \ \ \ \
\ \ \ \ \ \ \ \ \ \ \ \ \ \ \ \ \ \ \ \ \ \ \ \ \ \ \ \ \ \ \ \ \ \ \ \ \ \
\ \ \ \ \ \ \ \ \ \ \ 
\begin{equation*}
(1-\nu )\left( f\left( \frac{1}{2}\right) -f\left( \nu \right) \right) \leq
\left( \frac{1}{2}-\nu \right) \left( f\left( 1\right) -f\left( \nu \right)
\right) .
\end{equation*}%
$.$

Adding $f(1)$ to both sides, gives%
\begin{equation*}
f(1)-f(\nu )\leq 2(1-\nu )\left( f(1)-f\left( \frac{1}{2}\right) \right) .
\end{equation*}%
i.e.,%
\begin{equation*}
\left\vert \left\vert \left\vert AX+XB\right\vert \right\vert \right\vert
-\left\vert \left\vert \left\vert A^{\nu }XB^{1-\nu }+A^{1-\nu }XB^{\nu
}\right\vert \right\vert \right\vert \leq 2(1-\nu )\left( \left\vert
\left\vert \left\vert AX+XB\right\vert \right\vert \right\vert -2\left\vert
\left\vert \left\vert A^{\frac{1}{2}}XB^{\frac{1}{2}}\right\vert \right\vert
\right\vert \right) .
\end{equation*}

If $0<\frac{1}{2}<\nu <1$, then%
\begin{equation*}
\frac{f\left( \frac{1}{2}\right) -f\left( 0\right) }{\frac{1}{2}-0}\leq 
\frac{f(\nu )-f(0)}{\nu -0},
\end{equation*}%
and so%
\begin{equation*}
f\left( 0\right) -f\left( \nu \right) \leq 2\nu \left( f\left( 0\right)
-f\left( \frac{1}{2}\right) \right) ,
\end{equation*}%
i.e.,%
\begin{equation*}
\left\vert \left\vert \left\vert AX+XB\right\vert \right\vert \right\vert
-\left\vert \left\vert \left\vert A^{\nu }XB^{1-\nu }+A^{1-\nu }XB^{\nu
}\right\vert \right\vert \right\vert \leq 2\nu \left( \left\vert \left\vert
\left\vert AX+XB\right\vert \right\vert \right\vert -2\left\vert \left\vert
\left\vert A^{\frac{1}{2}}XB^{\frac{1}{2}}\right\vert \right\vert
\right\vert \right) .
\end{equation*}

Thus, from the above two norm inequalities, we get the inequality (\ref{ineq
23}).
\end{proof}

The inequalities (\ref{ineq 22}) and (\ref{ineq 23}) can be combined so that%
\begin{eqnarray}
2r_{0}\left( \left\vert \left\vert \left\vert AX+XB\right\vert \right\vert
\right\vert -2\left\vert \left\vert \left\vert A^{\frac{1}{2}}XB^{\frac{1}{2}%
}\right\vert \right\vert \right\vert \right) &\leq &\left\vert \left\vert
\left\vert AX+XB\right\vert \right\vert \right\vert -\left\vert \left\vert
\left\vert A^{\nu }XB^{1-\nu }+A^{1-\nu }XB^{\nu }\right\vert \right\vert
\right\vert  \notag \\
&\leq &2R_{0}\left( \left\vert \left\vert \left\vert AX+XB\right\vert
\right\vert \right\vert -2\left\vert \left\vert \left\vert A^{\frac{1}{2}%
}XB^{\frac{1}{2}}\right\vert \right\vert \right\vert \right) .
\label{ineq 24}
\end{eqnarray}

Applying Theorem \ref{T2.1} to the inequalities (\ref{ineq 24}), we have the
following general result.

\begin{corollary}
\label{C3.1}Let $A,B,X\in M_{n}(%
\mathbb{C}
)$\ such that $A$\ and $B$\ are positive semidefinite. If $\phi :[0,\infty
)\rightarrow 
\mathbb{R}
$ is a stricly increasing convex function, \ then we have%
\begin{eqnarray*}
&&\phi \left( 2r_{0}\left\vert \left\vert \left\vert AX+XB\right\vert
\right\vert \right\vert \right) -\phi \left( 4r_{0}\left\vert \left\vert
\left\vert A^{\frac{1}{2}}XB^{\frac{1}{2}}\right\vert \right\vert
\right\vert \right) \\
&\leq &\phi \left( \left\vert \left\vert \left\vert AX+XB\right\vert
\right\vert \right\vert \right) -\phi \left( \left\vert \left\vert
\left\vert A^{\nu }XB^{1-\nu }+A^{1-\nu }XB^{\nu }\right\vert \right\vert
\right\vert \right) \\
&\leq &\phi \left( 2R_{0}\left\vert \left\vert \left\vert AX+XB\right\vert
\right\vert \right\vert \right) -\phi \left( 4R_{0}\left\vert \left\vert
\left\vert A^{\frac{1}{2}}XB^{\frac{1}{2}}\right\vert \right\vert
\right\vert \right) ,
\end{eqnarray*}%
where $0\leq \nu \leq 1$, $r_{0}=\min \left\{ \nu ,1-\nu \right\} $, and \ $%
R_{0}=\max \left\{ \nu ,1-\nu \right\} $.
\end{corollary}

In particular, if $\phi (x)=x^{q}$ $\left( q\in 
\mathbb{R}
,q\geq 1\right) $ and $\left\vert \left\vert \left\vert .\right\vert
\right\vert \right\vert =\left\Vert .\right\Vert _{p}$ (the Schatten p-norm $%
p\in 
\mathbb{R}
,$ $p\geq 1$), we have 
\begin{eqnarray*}
(2r_{0})^{q}\left( \left\Vert AX+XB\right\Vert _{p}^{q}-2^{q}\left\Vert A^{%
\frac{1}{2}}XB^{\frac{1}{2}}\right\Vert _{p}^{q}\right) &\leq &\left\Vert
AX+XB\right\Vert _{p}^{q}-\left\Vert A^{\nu }XB^{1-\nu }+A^{1-\nu }XB^{\nu
}\right\Vert _{p}^{q} \\
&\leq &(2R_{0})^{q}\left( \left\Vert AX+XB\right\Vert
_{p}^{q}-2^{q}\left\Vert A^{\frac{1}{2}}XB^{\frac{1}{2}}\right\Vert
_{p}^{q}\right) .
\end{eqnarray*}

\end{document}